\documentclass[12pt]{article}
\DeclareMathAlphabet{\eusm}{OT1}{eusm}{m}{n}
\DeclareMathAlphabet{\Bbb}{OT1}{msbm}{b}{n} \textheight 22cm
\title { Symmetric Derivations on K\"{a}hler Modules}
\author{ Necati Olgun
\thanks{Department of Mathematics, University of Gaziantep, Gaziantep, Turkey olgun@gantep.edu.tr}}
\usepackage{amsmath}
\usepackage{amssymb}
\usepackage{amscd}
\date{}
\begin{document}
\maketitle

\newtheorem{thm}{Theorem}[section]
\newtheorem{lem}[thm]{Lemma}
\newtheorem{prop}[thm]{Proposition}
\newtheorem{cor}[thm]{Corollary}
\newtheorem{df}[thm]{Definition}
\newtheorem{nota}{Notation}
\newtheorem{note}[thm]{Remark}
\newtheorem{ex}[thm]{Example}
\newenvironment{proof}{\par\noindent{\bf Proof \,}}{$\hfill \Box$\par\bigskip}

\begin{abstract}
 In this paper, we define generalized symmetric derivations on K\"{a}hler modules. We give the relationships between the projective dimensions of $\Omega^{(1)}(R/k)$ and $\Omega^{(2)}(R/k)$ by using the symmetric derivation.We then give some interesting results by using this definition and related to K\"{a}hler modules and symmetric derivations.\\

\end{abstract}

\section{Introduction}\label{int}
\quad \quad

The concept of a K\"{a}hler module of $q$ th order was introduced by H.Osborn in 1965 [10]. Same notion has appeared in
R.G. Heyneman and M.E. Sweedler [12]. They introduced differential operators on a commutative algebra which extend the notion of derivations. J. Johnson [3] introduced differential module structures on certain modules of K\"{a}hler differentials.
Y. Nakai [7] developed the fundamental theories for the calculus of high order derivations and some functorial properties of the module of high order differentials in his paper. M.E. Sweedler [13] introduced right differential operators on a noncommutative algebra, which extend the notion of right derivations. Komatsu [4] introduced right differential operators on a noncommutative ring extension. In [8] author characterized the homological dimension of n-th order K\"{a}hler differentials of $R$ over $k$, and examined functorial properties of the module of n-th order K\"{a}hler differentials of $R\bigotimes S$ over $k$ in [9]. Then many authors studied on the properties of K\"{a}hler modules [2,5,14,15].

The purpose of this paper is introduced generalized (high order) symmetric derivation on high order K\"{a}hler modules which have not been considered before.
We will construct the generalized symmetric derivation on high order K\"{a}hler modules of commutative ring extension $R/k$
and show their fundamental properties. We give the relationships between the projective dimensions of $\Omega^{(1)}(R/k)$ and $\Omega^{(2)}(R/k)$ by using the symmetric derivation.  In particular our exact sequence of high order K\"{a}hler modules were not known related to regularity of the commutative rings.

Throughout this paper we will let $R$ be a commutative algebra over an algebraically closed field $k$ with characteristic
zero. When $R$ is a $k$-algebra,$J_{n}(R/k)$ or $J_{n}(R)$ denotes the universal module of n-th order differentials
of $R$ over $k$ and $\Omega^{(q)}(R/k)$ denotes the module
of q-th order K\"{a}hler differentials of $R$ over $k$ and
$\delta^{(q)}_{R/k}$ or $\delta^{(q)}$ denotes the canonical q-th
order $k$-derivation $R\rightarrow \Omega^{(q)}(R/k)$ of $R$. The
pair \{$\Omega^{(q)}(R/k),\delta^{(q)}_{R/k}$\} has the universal
mapping property with respect to the q-th order $k$-derivations of
$R$. $I_{R/k}$ or $I_{R}$ denotes the kernel of the canonical
mapping $R\otimes_{k}R\rightarrow R$ $(a\otimes b \rightarrow
ab)$. $\Omega^{(q)}(R/k)$ is identified with $I_{R}/I^{q+1}_{R}$.

It is well known that $J_{n}(R)\cong\Omega^{(n)}(R)\oplus R$.

$\Omega^{(q)}(R/k)$ is generated by the set $\{\delta^{(q)}(r): \
\ r \in R\}$. Hence if $R$ is finitely generated $k$-algebra,
$\Omega^{(q)}(R/k)$ will be a finitely generated $R$-module.\\

\section{Preliminaries}
In this section, we give some basic definitions and results about the symmetric derivation and K\"{a}hler modules of differentials.

\begin{df}[Osborn, 10]Let $R$ be any $k$-algebra (commutative with unit),
$R\rightarrow \Omega^{(q)}(R/k)$ be  $q$-th order K\"{a}hler
derivation of $R$ and let $S(\Omega^{(1)}(R/k))$ be the symmetric
algebra $\bigoplus_{p\geq 0}S^{p}(\Omega^{(1)}(R/k))$ generated over
$R$ by $\Omega^{(1)}(R/k)$.

A symmetric derivation is any linear map $D$ of
$S(\Omega^{(1)}(R/k))$ into itself such that

i) $D(S^{p}(\Omega^{(1)}(R/k)))\subset S^{p+1}(\Omega^{(1)}(R/k))$

ii) $D$ is a first order derivation over $k$ and

iii) the restriction of $D$ to $R$ ($R\simeq
S^{0}(\Omega^{(1)}(R/k))$) is the K\"{a}hler derivation $\delta^{(1)}:
R\rightarrow \Omega^{(1)}(R/k)$
\end{df}

\begin{thm}[Osborn, 10]Let $R$ be an affine $k$-algebra. Then there exists a
short exact sequence of $R$ modules
  $$0\rightarrow Ker\theta \rightarrow
\Omega^{(2)}(R/k)\overset{\theta}{\rightarrow}
\Omega^{(1)}(R/k)\rightarrow 0$$ such that
$\theta(\delta^{(2)}(f))=\delta^{(1)}(f) $ for all $f \in R$ and
$Ker\theta$ is generated by the set $\{
\delta^{(2)}(ab)-a\delta^{(2)}(b)-b\delta^{(2)}(a) \}$ for all
$a,b\in R$.
\end{thm}

\begin{prop}[Osborn, 10] $S^{2}(\Omega^{(1)}(R/k))\simeq Ker\theta $
\end{prop}

\begin{proof}:$\delta^{(1)}(a).\delta^{(1)}(b)=(1\otimes a-a\otimes 1).(1\otimes b-b\otimes 1)$

$= 1\otimes ab-b\otimes a - a\otimes b+ ab\otimes 1$

$=(1\otimes ab-ab\otimes 1)-a(1\otimes b-b\otimes 1)-b(1\otimes
a-a\otimes 1)$

$=\delta^{(2)}(ab)-a\delta^{(2)}(b)-b\delta^{(2)}(a)$ Then we have
$S^{2}(\Omega^{(1)}(R/k))\simeq Ker\theta $ as required.

\end{proof}

\begin{thm}[Sweedler, 11]Let $R$ be an affine $k$-algebra. If $R$ is a regular ring, then $\Omega^{(q)}(R/k)$ is a projective
$R$-module.
\end{thm}

\begin{thm}[McConnell and Rabson, 6]Let $R$ be an affine $k$-algebra. $R$ is regular ring if and only if $\Omega^{(1)}(R/k)$ is a projective
$R$-module.
\end{thm}

\section{The Generalized Symmetric Derivations}

\begin{df}Let $R$ be any $k$-algebra (commutative with unit),
$R\rightarrow \Omega^{(q)}(R/k)$ be  $q$-th order K\"{a}hler
derivation of $R$ and let $S(\Omega^{(q)}(R/k))$ be the symmetric
algebra $\bigoplus_{p\geq 0}S^{p}(\Omega^{(q)}(R/k))$ generated over
$R$ by $\Omega^{(q)}(R/k)$.

A generalized symmetric derivation is any $k$-linear map $D$ of
$S(\Omega^{(q)}(R/k))$ into itself such that

i) $D(S^{p}(\Omega^{(q)}(R/k)))\subset S^{p+1}(\Omega^{(q)}(R/k))$

ii) $D$ is a $q$-th order derivation over $k$ and

iii) the restriction of $D$ to $R$ ($R\simeq
S^{0}(\Omega^{(q)}(R/k))$) is the K\"{a}hler derivation $\delta^{(q)}:
R\rightarrow \Omega^{(q)}(R/k)$
\end{df}

\begin{ex}Let $R=k[x_{1},....,x_{s}]$ be a polynomial algebra of dimension s. Then
$\Omega^{(q)}(R/k)$ is a free $R$-module of rank
$(\begin{array}{c}
  q+s \\
  s
\end{array})-1$
with basis
$\{\delta^{(q)}(x_{1}^{i_{1}}.....x_{s}^{i_{s}}):i_{1}+.....+i_{s}\leq q\}$

$S^{2}(\Omega^{(q)}(R/k))$ is a free $R$-module of rank $(\begin{array}{c}
  t+1 \\
  t-1
\end{array})$

where $t=(\begin{array}{c}
  q+s \\
  s
\end{array})-1$
with basis
$\{\delta^{(q)}(x_{1}^{i_{1}}.....x_{s}^{i_{s}})\otimes \delta^{(q)}(x_{1}^{i_{1}}.....x_{s}^{i_{s}}):i_{1}+.....+i_{s}\leq
q\}$

\end{ex}

\begin{thm}Let $R$ be an affine $k$-algebra. Then there exists a
long exact sequence of $R$ modules
  $$0\rightarrow Ker\theta \rightarrow
\Omega^{(2q)}(R/k)\overset{\theta}{\rightarrow}
J_{q}(\Omega^{(q)}(R/k))\rightarrow Coker\theta \rightarrow 0$$ for all $q\geq 0$
\end{thm}

\begin{proof}: Let $R$ be any $k$-algebra,
$\Omega^{(q)}(R/k)$ be  $q$-th order K\"{a}hler
derivation of $R$. Let $J_{q}(\Omega^{(q)}(R/k))$ be  $q$-th order universal module of differential operators of order less than or equal to $q$ on $\Omega^{(q)}(R/k)$ with the universal differential operator $\Delta_{q}:\Omega^{(q)}(R/k)\rightarrow J_{q}(\Omega^{(q)}(R/k))$.

By the universal mapping property of $\Omega^{(2q)}(R/k)$ there exists a unique $R$-module homomorphism $\theta: \Omega^{(2q)}(R/k)\rightarrow J_{q}(\Omega^{(q)}(R/k))$ such that $\theta \delta^{2q}=\Delta_{q}\delta^{q}$ and the following diagram commutes.

   $$\begin{array}{ccc}
      R & \rightarrow^{\delta^{q}} & \Omega^{(q)}(R/k) \\
      \downarrow \delta^{2q} &   & \downarrow \Delta_{q}  \\
      \Omega^{(2q)}(R/k) & \rightarrow^{\theta} & J_{q}(\Omega^{(q)}(R/k)) \\
    \end{array}$$

 From homological properties we obtain the sequence as required.
\end{proof}

 \begin{cor} If $R$ be any regular $k$-algebra with dimension $s$ then $\theta$ is injective.
 \end{cor}

\begin{proof}: We define $\theta(\delta^{2q}(x^{\alpha} ))= \Delta_{q}( x^{\beta} \delta^{q}(x^{\gamma}))$

where

$x^{\alpha}=x_{1}^{i_{1}}.....x_{s}^{i_{s}}  : i_{1}+.....+i_{s}\leq 2q$

      $x^{\beta}=x_{1}^{i_{1}}.....x_{s}^{i_{s}}  : i_{1}+.....+i_{s}\leq q$

      $x^{\gamma}=x_{1}^{i_{1}}.....x_{s}^{i_{s}} : i_{1}+.....+i_{s}\leq q$
\end{proof}

\begin{cor}[Erdogan] If $R$ be any regular $k$-algebra with dimension $1$ and $q=1$ then $\Omega^{(2)}(R/k))$ is isomorphic to $J_{1}(\Omega^{(1)}(R/k))$.
\end{cor}

\begin{proof}: We have rank of the module $\Omega^{(2)}(R/k))$ is
$t=(\begin{array}{c}
  2+s \\
  s
\end{array})-1$ and it equals to rank of $J_{1}(\Omega^{(1)}(R/k))$.

But this is not true in other cases ($q>1$).[see 2]

\end{proof}

\begin{ex}Let $R=k[x,y]$ be a polynomial algebra of dimension 2. Then
$\Omega^{(1)}(R/k)$ is a free $R$-module of rank 2 with basis
$\{ \delta^{(1)}(x), \delta^{(1)}(y) \}$

$\Omega^{(2)}(R/k)$ is a free $R$-module of rank 5 with basis

$\{ \delta^{(2)}(x), \delta^{(2)}(y), \delta^{(2)}(x^2), \delta^{(2)}(xy), \delta^{(2)}(y^2)\}$

$J_{1}(\Omega^{(1)}(R/k))$ is a free $R$-module of rank 6 with basis

$\{ \Delta_{1}(\delta^{(1)}(x)), \Delta_{1}(\delta^{(1)}(y)), \Delta_{1}(x\delta^{(1)}(x)), \Delta_{1}(x\delta^{(1)}(y)), \Delta_{1}(y\delta^{(1)}(x)), \Delta_{1}(y\delta^{(1)}(y))\}$

$J_{2}(\Omega^{(2)}(R/k))$ is a free $R$-module of rank 30 with basis

$\{ \Delta_{2}(\delta^{(2)}(x)), \Delta_{2}(\delta^{(2)}(y)), \Delta_{2}(\delta^{(2)}(x^2)), \Delta_{2}(\delta^{(2)}(xy)), \Delta_{2}(\delta^{(2)}(y^2)),$

$\Delta_{2}(x\delta^{(2)}(x)), \Delta_{2}(x\delta^{(2)}(y)), \Delta_{2}(x\delta^{(2)}(x^2)), \Delta_{2}(x\delta^{(2)}(xy)), \Delta_{2}(x\delta^{(2)}(y^2)),$

$\Delta_{2}(y\delta^{(2)}(x)), \Delta_{2}(y\delta^{(2)}(y)), \Delta_{2}(y\delta^{(2)}(x^2)), \Delta_{2}(y\delta^{(2)}(xy)), \Delta_{2}(y\delta^{(2)}(y^2)),$

$\Delta_{2}(x^2\delta^{(2)}(x)), \Delta_{2}(x^2\delta^{(2)}(y)), \Delta_{2}(x^2\delta^{(2)}(x^2)), \Delta_{2}(x^2\delta^{(2)}(xy)), \Delta_{2}(x^2\delta^{(2)}(y^2)),$

$\Delta_{2}(xy\delta^{(2)}(x)), \Delta_{2}(xy\delta^{(2)}(y)), \Delta_{2}(xy\delta^{(2)}(x^2)), \Delta_{2}(xy\delta^{(2)}(xy)), \Delta_{2}(xy\delta^{(2)}(y^2)),$

$\Delta_{2}(y^2\delta^{(2)}(x)), \Delta_{2}(y^2\delta^{(2)}(y)), \Delta_{2}(y^2\delta^{(2)}(x^2)), \Delta_{2}(y^2\delta^{(2)}(xy)), \Delta_{2}(y^2\delta^{(2)}(y^2)),\}$

$S^{2}(\Omega^{(1)}(R/k))$ is a free $R$-module of rank 3 with
basis $\{
\delta^{(1)}(x)\otimes\delta^{1}(x),\delta^{(1)}(x)\otimes\delta^{(1)}(y),\delta^{(1)}(y)\otimes\delta^{(1)}(y)
\}$.

In this example $\Omega^{(2)}(R/k))$ is not isomorphic to $J_{1}(\Omega^{(1)}(R/k))$ and we obtain the exact sequence
$$0\rightarrow
S^{2}(\Omega^{(1)}(R/k)) \rightarrow \Omega^{(2)}(R/k)
{\rightarrow} \Omega^{(1)}(R/k)\rightarrow 0$$ of $R$ modules.

\end{ex}

\begin{thm}Let $R$ be an affine $k$-algebra. Then there exists a
long exact sequence of $R$ modules
  $$0\rightarrow Ker\beta \rightarrow J_{q}(\Omega^{(q)}(R/k))\overset{\beta}{\rightarrow}
S^2(\Omega^{(q)}(R/k))\rightarrow Coker\beta \rightarrow 0$$ for all $q\geq 0$

\end{thm}

\textbf{Proof:} Let $D_{q}:\Omega^{(q)}(R/k))\rightarrow S^2(\Omega^{(q)}(R/k))$ be any generalized symmetric derivation and let $J_{q}(\Omega^{(q)}(R/k))$ be  $q$-th order universal module of differential operators of order less than or equal to $q$ on $\Omega^{(q)}(R/k)$ with the universal differential operator $\Delta_{q}:\Omega^{(q)}(R/k)\rightarrow J_{q}(\Omega^{(q)}(R/k))$.

By the universal mapping property of $J_{q}(\Omega^{(q)}(R/k))$ there exists a unique $R$-module homomorphism $\beta:J_{q}(\Omega^{(q)}(R/k))\rightarrow S^2(\Omega^{(q)}(R/k))$ such that the following diagram commutes.

   $$\begin{array}{ccc}
      \Omega^{(q)}(R/k) & \rightarrow^{D_{q}} & S^2(\Omega^{(q)}(R/k)) \\
      \downarrow \Delta_{q} &   & \downarrow   \\
      J_{q}(\Omega^{(q)}(R/k)) & \rightarrow^{\beta} & S^2(\Omega^{(q)}(R/k)) \\
    \end{array}$$

 From homological properties we have the sequence as required.

\begin{lem}:Let $R$ be an affine domain with dimension $s$.  Then $\Omega^{(q)}(R/k)$ is a free $R$-module if and only if $S^{2}(\Omega^{(q)}(R/k))$ is a free $R$-module.
\end{lem}

\begin{proof}: Without loss of generality we may assume that $R$ is local domain of dimension s.
Suppose that $\Omega^{(q)}(R/k)$ is free $R$-module. From the property of symmetric algebra
$S^{2}(\Omega^{(q)}(R/k))$ is a free $R$-module.

Conversly, suppose that $S^{2}(\Omega^{(q)}(R/k))$ is a free
$R$-module. If $dimR=s$ then the rank of $\Omega^{(q)}(R/k)$ is
$(\begin{array}{c}
  q+s \\
  s
\end{array})-1$. Let $(\begin{array}{c}
  q+s \\
  s
\end{array})-1=t$. Then the rank of $S^{2}(\Omega^{(q)}(R/k))$ is $(\begin{array}{c}
  t+1 \\
  t-1
\end{array})$.

Let $m$ be the maximal ideal of $R$. Then
$S^{2}(\Omega^{(q)}(R/k))\otimes_{R}R/m$ is an $R/m$ vector space
of dimension $(\begin{array}{c}
  t+1 \\
  t-1
\end{array})$.
$S^{2}(\Omega^{(q)}(R/k))\otimes_{R}R/m $ is isomorphic to
$S^{2}$(${\Omega^{(q)}(R/k)}\over{m\Omega^{(q)}(R/k)}$). Then
$S^{2}$(${\Omega^{(q)}(R/k)}\over{m\Omega^{(q)}(R/k)}$)is an $R/m$
vector space of dimension $(\begin{array}{c}
  t+1 \\
  t-1
\end{array})$ if and only if ${\Omega^{(q)}(R/k)}\over{m\Omega^{(q)}(R/k)}$ is an $R/m$
vector space of dimension $t$. Hence
${\Omega^{(q)}(R/k)}\over{m\Omega^{(q)}(R/k)}$ is an $R/m$ vector
space of dimension $t$ if and only if the number of minimal
generators of $\Omega^{(q)}(R/k)$ is $t$. The rank of
$\Omega^{(q)}(R/k)$ was $t$. Therefore it is obtained
$\Omega^{(q)}(R/k)$ is a free $R$-module as required.
\end{proof}

\begin{thm}: Let $R$ be an affine $k$-algebra and
$S(\Omega^{(1)}(R/k))$ has at least one symmetric derivation.
$\Omega^{(1)}(R/k)$ is a projective $R$-module if and only if
$\Omega^{(2)}(R/k)$ is a projective $R$-module.
\end{thm}

\begin{proof}: Suppose that $\Omega^{(1)}(R/k)$ is a projective
$R$-module. By Theorem 2.5 $R$ is a regular ring and $\Omega^{(2)}(R/k)$
is a projective $R$-module by Theorem 2.4.

Conversly, From by Proposition 2.3. and Theorem 2.2 we have the
exact sequence
$$0\rightarrow
S^{2}(\Omega^{(1)}(R/k)) \rightarrow \Omega^{(2)}(R/k)
{\rightarrow} \Omega^{(1)}(R/k)\rightarrow 0$$
of $R$-modules. This sequence is split exact sequence. The splitting of the sequence is sufficient to prove.
It follows from the definition of the symmetric derivation any symmetric derivation $D$ is uniquely determined by its
restriction $D_{1}:\Omega^{(1)}(R/k)\rightarrow
S^{2}(\Omega^{(1)}(R/k))$, and it is observed that the composition
$D_{1} \delta^{(1)}$  $$ R\rightarrow \Omega^{(1)}(R/k)\rightarrow
S^{2}(\Omega^{(1)}(R/k))   $$ is a second order derivation of $R$.
By the universal mapping property of $\{\Omega^{(2)}(R/k),\delta^{(2)} \}$ there is unique $R$ module
homomorphism $t: \Omega^{(2)}(R/k)\rightarrow
S^{2}(\Omega^{(1)}(R/k))$ such that
$t\delta^{(2)}=D_{1}\delta^{(1)}$
$$t(\delta^{(2)}(ab)-a\delta^{(2)}(b)-b\delta^{(2)}(a))=t\delta^{(2)}(ab)-at\delta^{(2)}(b)-bt\delta^{(2)}(a)  $$
$$=D_{1}\delta^{(1)}(ab)-aD_{1}\delta^{(1)}(b)-bD_{1}\delta^{(1)}(a)$$
$$=D_{1}(a\delta^{(1)}(b)+b\delta^{(1)}(a))-aD_{1}\delta^{(1)}(b)-bD_{1}\delta^{(1)}(a)$$
$$=aD_{1}\delta^{(1)}(b)+ \delta^{1}(a)\delta^{(1)}(b)+bD_{1}\delta^{(1)}(a))+\delta^{1}(b)\delta^{(1)}(a)-aD_{1}\delta^{(1)}(b)-bD_{1}\delta^{(1)}(a)$$
$$=2\delta^{(1)}(a).\delta^{(1)}(b)$$
It follows that $(1/2)t$ splits as required.

Then $\Omega^{(2)}(R/k)$ is isomorphic to $\Omega^{(1)}(R/k)\bigoplus S^{2}(\Omega^{(1)}(R/k))$. Therefore
$\Omega^{(1)}(R/k)$ is a projective $R$-module.
\end{proof}

For an affine local $k$ algebra we give the following result.
\begin{cor}: Let $R$ be an affine local $k$-algebra and
$S(\Omega^{(1)}(R/k))$ has at least one symmetric
derivation.$\Omega^{(1)}(R/k)$ is a free $R$ module if and only if
$\Omega^{(2)}(R/k)$ is a free $R$ module.
\end{cor}

\begin{thm}: Let $R$ be an affine $k$-algebra and
$S(\Omega^{(1)}(R/k))$ has at least one symmetric derivation. $R$ is a
regular ring  if and only if $\Omega^{(2)}(R/k)$ is a projective
$R$ module.
\end{thm}

\begin{proof}: From by Theorem 2.4.,  Theorem 2.5. and Theorem 3.9.
\end{proof}

\begin{cor}: Let $R$ be an affine $k$-algebra and
$S(\Omega^{(1)}(R/k))$ has at least one symmetric derivation. $R$ is a
regular local ring  if and only if $\Omega^{(2)}(R/k)$ is a free
$R$ module.
\end{cor}

Now, it is obtained the following important results related to projective dimensions of K\"{a}hler modules by using the split exact sequence
$$0\rightarrow
S^{2}(\Omega^{(1)}(R/k)) \rightarrow \Omega^{(2)}(R/k)
{\rightarrow} \Omega^{(1)}(R/k)\rightarrow 0$$
of $R$-modules in the proof of Theorem 3.9. and homological properties.

\begin{cor}: Let $R$ be an affine $k$-algebra and
$S(\Omega^{(1)}(R/k))$ has at least one symmetric derivation. If the projective
dimension of  $\Omega^{(2)}(R/k)$ is finite then the projective
dimension of  $\Omega^{(1)}(R/k)$ is finite.
\end{cor}

\begin{proof}: $\Omega^{(2)}(R/k)$ is isomorphic to $\Omega^{(1)}(R/k)\bigoplus S^{2}(\Omega^{(1)}(R/k))$ from the split exact sequence $$0\rightarrow
S^{2}(\Omega^{(1)}(R/k)) \rightarrow \Omega^{(2)}(R/k)
{\rightarrow} \Omega^{(1)}(R/k)\rightarrow 0$$
of $R$-modules. Therefore it is obtained as required.
\end{proof}

\begin{cor}: Let $R$ be an affine $k$-algebra and
$S(\Omega^{(1)}(R/k))$ has at least one symmetric derivation. If the projective
dimension of  $\Omega^{(1)}(R/k)$ is infinite then the projective
dimension of  $\Omega^{(2)}(R/k)$ is infinite.
\end{cor}

\begin{ex}Let $S$ be the coordinate ring of the cups $y^2=x^3$. Then $S= k[x,y]/(f)$ where $f=y^2-x^3$.
It can be found the projective dimension of $\Omega^{(1)}(S/k)$, $\Omega^{(2)}(S/k)$ and $J_{(1)}(\Omega^{(1)}(S/k))$

$\Omega^{(1)}(S/k)\simeq F/N$ where $F$ is a free $S$ module on
$\{\delta^{1}(x),\delta^{1}(y)\}$ and $N$ is a submodule of $F$
generated by $\delta^{1}(f)=2y\delta^{1}(y)-3x^{2}\delta^{1}(x)$. Certainly
$N$ is a free on $\delta^{1}(f)$.  Therefore we have
$$0\longrightarrow N \overset{\phi}{\longrightarrow} F
\overset{\pi}{\longrightarrow} \Omega^{(1)}(S/k)\simeq F/N\longrightarrow
0$$
a free resolution of $\Omega^{(1)}(S/k)$. In this sequence the homomorphism $\phi$ is a matrix $\left(%
\begin{array}{c}
  -3x^2 \\
  2y \\
\end{array}%
\right) $  and projective dimension of $\Omega^{(1)}(S/k)$ less than or equal to one.

By the same argument $\Omega^{(2)}(S/k)\simeq F'/N'$ where $F'$ is a free
$S$ module on
$$\{\delta^{2}(x),\delta^{2}(y),\delta^{2}(xy),\delta^{2}(x^2),\delta^{2}(y^2)
\}$$ and $N'$ is a submodule of $F'$ generated by
$\{\delta^{2}(f),\delta^{2}(xf),\delta^{2}(yf)\}$ where

$\delta^{2}(f)=\delta^{2}(y^2)-3x\delta^{2}(x^2)+3x^2\delta^{2}(x)$

$\delta^{2}(xf)=x\delta^{2}(y^2)-6x^2\delta^{2}(x^2)+2y\delta^{2}(xy)+7x^3\delta^{2}(x)-2xy\delta^{2}(y)$

$\delta^{2}(yf)=3y\delta^{2}(y^2)-3xy\delta^{2}(x^2)-3x^2\delta^{2}(xy)+6x^2y\delta^{2}(x)-y^2\delta^{2}(y)$\\

Since $rank \Omega^{(2)}(S/k)=2$ we have $rank N'=rank F'-rank \Omega^{(2)}(S/k)=5-2=3$ . So $N'$ is  free $S$-module. Therefore we have
$$0\longrightarrow N' \overset\phi{\longrightarrow} F'\overset\pi{\longrightarrow} \Omega^{(2)}(S/k)\simeq F'/N'\longrightarrow 0$$
a free resolution of $\Omega^{(2)}(S/k)$. Here $\pi$ is the natural
surjection and $\phi$ is given by the following matrix

$$\left(%
\begin{array}{cccccc}
  -3x & 1 & 0 & 3x^2 & 0  \\
  -6x^2 & x & 2y & 7x^3 & -2xy   \\
  -3xy & 3y & -3x^2 & 6x^2y & -y^2  \\
\end{array}%
\right)$$  and projective dimension of $\Omega^{(2)}(S/k)$ less than or equal to one.

Similarly $J_{1}(\Omega^{(1)}(S/k))\simeq F''/N''$ where $F''$ is a free
$S$-module with basis $\{ \Delta_{1}(\delta^{(1)}(x)), \Delta_{1}(\delta^{(1)}(y)), \Delta_{1}(x\delta^{(1)}(x)), \Delta_{1}(x\delta^{(1)}(y)), \Delta_{1}(y\delta^{(1)}(x))\}$
and $N''$ is a submodule of $F''$ generated by

$a=3x^2 \Delta_{1}(x\delta^{(1)}(x))-2y\Delta_{1}(x\delta^{(1)}(y))-3x^3\Delta_{1}(\delta^{(1)}(x))+2xy\Delta_{1}(\delta^{(1)}(y))$

$b=3x^2 \Delta_{1}(x\delta^{(1)}(x))-2y\Delta_{1}(y\delta^{(1)}(x))-x^3\Delta_{1}(\delta^{(1)}(x))$

$c=-6xy \Delta_{1}(x\delta^{(1)}(x))+3x^2\Delta_{1}(x\delta^{(1)}(y))+3x^2y\Delta_{1}(\delta^{(1)}(x))+x^3\Delta_{1}(\delta^{(1)}(y))$

Since $rank J_{1}(\Omega^{(1)}(S/k))=2$ we have $rank N''=rank F''-rank
J_{1}(\Omega^{(1)}(S/k))=5-2=3$. So $N''$ is a free $S$-module of rank 3.
Therefore
$$0\longrightarrow N'' \longrightarrow F''\overset\pi{\longrightarrow} J_{1}(\Omega^{(1)}(S/k))\simeq F''/N''\longrightarrow 0$$
a free resolution of $J_{1}(\Omega^{(1)}(S/k))$ and projective dimension of $J_{1}(\Omega^{(1)}(S/k))$ less than or equal to one..

$\Omega^{(1)}(S/k)\simeq F/N$ where $F$ is a free $S$ module on $\{\delta^{1}(x),\delta^{1}(y)\}$ and $N$ is a submodule of $F$
generated by $\delta^{1}(f)=2y\delta^{1}(y)-3x^{2}\delta^{1}(x)$. Using this modules, $S^{2}(\Omega^{(1)}(S/k))\simeq S^{2}(F)/l_{N}$ where $S^2(F)$ is a free
module with basis

$\{ \delta^{(1)}(x)\bigvee \delta^{(1)}(x), \delta^{(1)}(x)\bigvee \delta^{(1)}(y),\delta^{(1)}(y)\bigvee \delta^{(1)}(y)\}$
and $l_{N}$ is a submodule of $S^2(F)$ generated by

$\delta^{1}(f)\bigvee\delta^{1}(x) =(2y\delta^{1}(y)-3x^{2}\delta^{1}(x))\bigvee\delta^{1}(x)=2y\delta^{1}(y)\bigvee\delta^{1}(x)-3x^{2}\delta^{1}(x)\bigvee\delta^{1}(x)$

$\delta^{1}(f)\bigvee\delta^{1}(y) =(2y\delta^{1}(y)-3x^{2}\delta^{1}(x))\bigvee\delta^{1}(y)=2y\delta^{1}(y)\bigvee\delta^{1}(y)-3x^{2}\delta^{1}(x)\bigvee\delta^{1}(y)$

Since $rank S^{2}(\Omega^{(1)}(S/k))=1$ we have $rank l_{N}=rank S^{2}(F)-rank
S^{2}(\Omega^{(1)}(S/k))=3-1=2$. So $l_{N}$ is a free $S$-module of rank 2.
Therefore
$$0\longrightarrow l_{N} \longrightarrow S^{2}(F)\overset\pi{\longrightarrow} S^{2}(\Omega^{(1)}(S/k))\simeq S^{2}(F)/l_{N}\longrightarrow 0$$
a free resolution of $S^{2}(\Omega^{(1)}(S/k))$ and projective dimension of $S^{2}(\Omega^{(1)}(S/k))$ less than or equal to one..
\end{ex}
\begin{ex}Let $R=k[x,y,z]$ with $y^2=xz$ and $z^2=x^3$. It can be found the projective dimension of $\Omega^{(1)}(S/k)=1$ and the projective dimension of $\Omega^{(2)}(S/k)$ is infinite.
\end{ex}
\section{Conclusions}
In this work we have introduced the concept of higher symmetric derivation on K\"{a}hler modules and studied some of its properties. An application of this theory is given in solving a projective dimension problem. Consequently, it is not defined a symmetric derivation for an arbitrary ring (see Corollary 3.13 and example 3.16). Now we produce the following two problems for further works,

First problem: Can we define a symmetric derivation on which rings?

Second problem: Can we find a relationship between the projective dimensions of $\Omega^{(1)}(R/k)$ and $\Omega^{(n)}(R/k)$?

        \end{document}